\documentclass[10pt, oneside]{amsart}
\usepackage[hscale = 0.8, top=0.8in]{geometry}
\usepackage{appendix}

\usepackage{xcolor}
\usepackage{float}
\usepackage{amsmath}      
\usepackage{amsthm}        
\usepackage{amssymb}       
\usepackage{empheq}       
\usepackage{stackrel}
\usepackage{braket}
\usepackage[all]{xy}
\usepackage{hyperref}
\hypersetup{
  colorlinks   = true,          
  urlcolor     = pink,          
  linkcolor    = red, 
  citecolor   = blue            
}

\makeatletter
\newcommand{\makesmaller}[2]{#1{\mathpalette\make@smaller@{#2}}}
\newcommand{\make@smaller@}[2]{%
  \vcenter{\hbox{$\m@th\downgrade@style#1#2$}}%
}
\newcommand{\downgrade@style}[1]{%
  \ifx#1\displaystyle\scriptstyle\else
    \ifx#1\textstyle\scriptstyle\else
      \scriptscriptstyle
  \fi\fi
}
\makeatother

\renewcommand{\theequation}{\thesubsubsection.\arabic{equation}}         
\makeatletter                                                      
   \@addtoreset{equation}{subsubsection} 
   \@addtoreset{footnote}{subsubsection}

\def\newprooflikeenvironment#1#2#3#4{%
     \newenvironment{#1}[1][]{%
         \refstepcounter{equation}%
         \begin{proof}[{\rm\csname#4\endcsname{#2~\theequation}}]%
         \def\qedsymbol{#3}}%
        {\end{proof}}}                                             
\makeatother                                                   
                                                                   
\newprooflikeenvironment{definition}{Definition}{}{textbf}
\newprooflikeenvironment{example}{Example}{}{textbf}     
\newprooflikeenvironment{remark}{Remark}{}{textbf}       
\newprooflikeenvironment{notation}{Notaion}{}{textbf}
\newprooflikeenvironment{observation}{Observation}{}{}
\newprooflikeenvironment{quasiobservation}{}{}{}
                                                                   
\theoremstyle{plain}                                               
\newtheorem{theorem}[subsubsection]{Theorem}                            
                          
\newtheorem{lemma}[subsubsection]{Lemma}                                
\newtheorem{corollary}[subsubsection]{Corollary}                        
\newtheorem{proposition}[subsubsection]{Proposition}

\newcommand{\arr}{\longrightarrow}

\newcommand{\calg}{\mathrm{CAlg}}

\newcommand{\bs}{{\makesmaller{\mathrel}{\blacksquare}}}

\newcommand{\CC}{\mathcal{C}}
\newcommand{\comp}{{\makesmaller{\mathrel}{\wedge}}}

\newcommand{\colim}{\mathrm{colim}}
\newcommand{\DD}{\mathcal{D}}
\newcommand{\define}{\text{def}}
\renewcommand{\dim}{\mathrm{dim}}
\newcommand{\E}{\mathcal{E}}

\newcommand{\fib}{\mathrm{fib}}

\newcommand{\Ga}{\Gamma}
\renewcommand{\hom}{\text{Hom}}

\newcommand{\rig}{\mathrm{rig}}

\newcommand{\ind}{\mathrm{Ind}}
\newcommand{\K}{\mathcal{K}}

\newcommand{\leftcomp}{{\stackrel{\comp}{\text{\phantom{a}}}\!\!}}

\renewcommand{\mod}{\mathrm{Mod}}
\newcommand{\N}{\mathbb{N}}
\newcommand{\nuc}{\mathrm{Nuc}}
\renewcommand{\O}{\mathcal{O}}

\newcommand{\op}{\mathrm{op}}
\newcommand{\p}{\mathfrak{p}}

\newcommand{\qcoh}{\text{QCoh}}

\newcommand{\q}{\mathfrak{q}}

\DeclareMathOperator{\spec}{Spec}

\DeclareMathOperator{\uhom}{\underline{Hom}}

\newcommand{\Z}{\mathbb{Z}}

\newcommand{\invsimeq}{\mathrel{\rotatebox[origin=c]{90}{$\simeq$}}}

\newcommand{\invertin}{\mathrel{\rotatebox[origin=c]{90}{$\in$}}}

\newcommand{\ssec}{\subsection}
\newcommand{\sssec}{\subsubsection}
\renewcommand{\thicksim}{{\overset{\sim}{\!\!\phantom{\ldots}}}}

\newcommand{\pr}{\mathrm{Pr}}
\newcommand{\st}{\mathrm{st}}

\newcommand{\dual}{\mathrm{dual}}

\title{Adelic descent for continuous localizing invariants}
\author{Grigorii Konovalov}
\date{July 2025}

\begin{document}

\begin{abstract}
We prove a version of adelic descent for continuous localizing invariants.
\end{abstract}
\maketitle

\section{Introduction}

Let $X$ be a scheme of finite type over $\Z$. Recall from \cite{brav2024beilinsonparshinadelessolidalgebraic},
Theorem 4.2.4, that solid quasi-coherent sheaves on $X$ satisfy adelic descent. Namely, for any tuple $0 \le k_0 < \ldots k_m \le d$,
where $d$ denotes the dimension of $X$, we have the corresponding adelic ring \[\begin{split}
L_{k_0} \ldots L_{k_m} \O_X \;\simeq\; \prod_{\substack{\p_m \in X \\ \dim \overline{\{\p_m\}} = k_m}} \leftcomp\: (\ldots \prod_{\substack{\p_1\in \overline{\{\p_2\}}\\\dim \overline{\{\p_1\}} = k_1}}\leftcomp\:(\prod_{\substack{\p_0 \in \overline{\{\p_1\}} \\ \dim \overline{\{\p_0\}} = k_0}} \leftcomp\:\O_{X,\p_0})_{\p_1} \ldots )_{\p_m}
\;\in\; \calg(\qcoh(X_\bs)) \;,
\end{split}\]
where $\leftcomp\:(-)_\p$ denotes the functor given by localizing at the prime $\p \in X$ first and then completing at it.
In fact, the formation of adelic rings defines a functor \[\xymatrix{
[0 < 1]^d \ar[rr] && \calg(\qcoh(X_\bs)) \\
0 \le k_0 < \ldots k_m \le d \ar@{}[u]|\invertin \ar@{|->}[rr] && L_{k_0} \ldots L_{k_m} \O_X \ar@{}[u]|\invertin &.
}\] In particular, we get a symmetric monoidal functor
\begin{equation}\label{intro-solid-adelic-functor}
\qcoh(X_\bs) \arr \stackrel[\substack{0 \le k_0 < k_1 < \ldots < k_m \le d \\ 0\le m}]{}{\lim} L_{k_0}\ldots L_{k_m} \O_X - \mod_{\qcoh(X_\bs)} \;,
\end{equation}
which happens to be an equivalence---see Theorem 4.2.4 of \cite{brav2024beilinsonparshinadelessolidalgebraic}.

All of the categories appearing in (\ref{intro-solid-adelic-functor}) are compactly generated, but they are not suitable for
studying localizing invariants because the category $\qcoh(X_\bs)^\omega$ satisfies a version of the Eilenberg swindle. Namely,
for any fixed object $P \in \qcoh(X_\bs)^\omega$, there exists a countable product $\prod_\N P \in \qcoh(X_\bs)^\omega$.
Nevetheless, Clausen-Scholze defined (8th lecture in \cite{complex}, 13th lecture in \cite{analytic})
the category of nuclear modules over an analytic ring
(more generally, nuclear objects in a symmetric monoidal category), which is a presentable stable dualizable category and is
expected to produce correct localizing invariants (in the sense of \cite{efimov2025ktheorylocalizinginvariantslarge})
in many situations including those coming from formal geometry.
For a certain nice class of localizing invariants such as K-theory, that expectation was proven by Efimov 
(\cite{efimov2025localizinginvariantsinverselimits}, Theorem 7.9) in the case of a ring completion.
It therefore should be reasonable to study localizing invariants of the adeles using the nuclear modules of Clausen-Scholze.

In this paper, we prove that the diagram \[\xymatrix{
[0<1]^d \ar[rr] && \calg(\pr^\dual_\st) \\
0 \le s_0 < s_1 < \ldots < s_m \le d \ar@{}[u]|\invertin \ar@{|->}[rr] && \nuc(L_{s_0} \ldots L_{s_m} \O_X, X_\bs) \ar@{}[u]|\invertin &,
}\]
which we basically get from the diagram (\ref{intro-solid-adelic-functor}) by passing to nuclear objects,
is mapped to a limit diagram by any stable localizing invariant, see Theorem \ref{thm-nuc-adelic-descent}.
We also prove that this diagram is a limit diagram
in both $\calg(\pr^L_\st)$ and $\calg(\pr^\dual_\st)$, see Theorem \ref{thm-adelic-descent-cats}.
The latter in fact follows directly from
Theorem 4.2.4 of \cite{brav2024beilinsonparshinadelessolidalgebraic} by passing to rigidification, but the argument
presented in this paper is in a sense independent of that theorem.

We note that versions of adelic descent for localizing invariants have already been recorded in the literature,
see \cite{kim2021adelicdescentktheory} and \cite{kim2024formalgluingdiagramscontinuous}. In the former paper, the author
studies localizing invariants of perfect modules over the adeles and proves the adelic descent in that context. In the latter
paper, the author does all the algebra within the category $\qcoh(X) - \mod(\pr^\dual_\st)$ and proves a version
of adelic descent for localizing invariants of dualizable categories. Namely, for a category $\CC \in \qcoh(X) - \mod(\pr^\dual_\st)$,
\cite{kim2024formalgluingdiagramscontinuous} constructs a diagram consisting of products of iterated completed localizations of $\CC$
computed inside $\qcoh(X) - \mod(\pr^\dual_\st)$ and proves the descent statement for it, see Theorem 1.4 of loc. cit..
On the other hand, our approach is to do all the algebra inside $\qcoh(X_\bs)$
(which has already been done in \cite{brav2024beilinsonparshinadelessolidalgebraic})
and see what kind statement about localizing invariants we can get as an output.

At the categorical level, the difference between our version and Kim's version (\cite{kim2024formalgluingdiagramscontinuous})
is in a sense similar to the difference between Clausen-Scholze's category of nuclear modules and Efimov's category of
nuclear modules in the case of a ring completion,
which is well explained in the introduction to \cite{efimov2025localizinginvariantsinverselimits}.
That difference may disappear after passing
to a nice localizing invariant---like it does in the case of a ring completion,
see Theorem 7.9 in \cite{efimov2025localizinginvariantsinverselimits}---but that seems to need a proof.

\medskip
\noindent \textbf{The paper is structured as follows.}
In the second section, we summarize some results about nuclear objects and derive some relevant corollaries about nuclear modules
over the adelic rings. In the thrid section, we state and prove out main theorem.

\ssec{Notation and conventions}

\sssec{}
In this paper, we work in the setting of $\infty$-categories. We adopt the usual convention that the word category means
$(\infty,1)$-category, and all the standard categorical notions and constructions, such as functors and (co)limits, should be
understood in this context. In particular, we shall typically refer to derived completion simply as completion.

\sssec{}
We will denote by $\Pr^L_\st$ the $(\infty,1)$-category of stable presentable categories and continuous
(colimit preserving) functors.
We will denote by $\pr^\dual_\st$ the $(\infty, 1)$-category of stable presentable dualizable categories and strongly continuous functors.

\sssec{} By default, all limits of (stable presentable) categories are computed inside $\pr^L_\st$. Limits computed inside
$\pr^\dual_\st$ will be denoted by $\stackrel{\dual}{\lim}$.

\sssec{} Because we need the results of Appendix A, \cite{brav2024beilinsonparshinadelessolidalgebraic},
we work in finite type over $\Z$. It is also possible to work in finite type over a countable field
using the ultrasolid vector spaces over the field as the base category.

\sssec{} We use notation of \cite{brav2024beilinsonparshinadelessolidalgebraic} for completions. It is summarized in the next four points.

\sssec{}
Let $Z \subseteq \spec(A)$ be a closed subset defined by an ideal $I \subset A$. We will use $A^\comp_I$
to denote the (derived) completion of $A$ at the ideal $I$. 
The same goes for the completion of a (solid) module over $A$.
We note that the functor of completion $(-)^\comp_I$ at the ideal $I \subset A$ only depends on the closed subset $Z \subseteq \spec(A)$, but
not on the ideal $I$ cutting out $Z$. To emphasize that, and for other reasons, we will often call the functor of completion
at $I$ the functor of completion at $Z$ and denote it $(-)^\thicksim_Z$.

\sssec{}
More generally, for a specialization closed subset $T \subseteq \spec(A)$ (which can be thought of as a subset equal to the union
of all its subsets which are closed in $\spec(A)$), we will define a completion at $T$ and denote it $(-)^\thicksim_T$,
see section 1.4 in \cite{brav2024beilinsonparshinadelessolidalgebraic}.

\sssec{}
Similarly, for a generalization closed subset $E \subset \spec(A)$, we will define a completion at $E$ and denote it
$(-)^\thicksim_E$, see section 1.7 in \cite{brav2024beilinsonparshinadelessolidalgebraic}.
This will turn out to be equivalent to the completion
with respect to the ultrasolid analytic ring structure on the ring \[
\Ga(E, \O) \;\stackrel{\define}{=}\; \stackrel[U \supseteq E]{}{\colim} \Ga(U, \O) \;,
\] where the colimit is taken over all Zariski open subsets $U \subseteq \spec(A)$ containing $E$ and ordered by (opposite) inclusion.
See Definition 1.1.5.2 in \cite{brav2024beilinsonparshinadelessolidalgebraic}
for the definition of the ultrasolid analytic ring structure.

\sssec{}
For a prime ideal $\p \subset A$, we will denote by $\leftcomp A_\p$ the local ring $A_\p$ completed at the maximal ideal.
Similarly, for any (solid) $A$-module $M$, $\leftcomp\: M_\p$ denotes the localization of $M$ at $\p$ followed by the (derived) completion
at the maximal ideal of $A_\p$.

\sssec{}
We use Lurie's notation for categories of modules over an algebra as opposed to the traditional derived category notation. For example,
$\mod_A$ means the whole stable $(\infty, 1)$-category of modules over the algebra $A$. Likewise, $\mod_{A_\bs}$ denotes
the stable $(\infty, 1)$-category of solid $A$-modules.

This also applies to categories of quasi-coherent sheaves over a scheme: $\qcoh(X)$ means the stable $(\infty,1)$-category of
quasi-coherent complexes over $X$; $\qcoh(X_\bs)$ means the stable $(\infty,1)$-category of solid quasi-coherent sheaves over $X$.

\sssec{} For concreteness, we work in light setting of the condensed mathematics, see \cite{clausen2023analytic}, Lecture 2.
Specifically, this means that,
for a commutative algebra $A$ of finite type over $\Z$, the category of solid $A$-modules $\mod_{A_\bs}$ is compactly generated
by a single compact generator \[
\prod_\N A \;\in\; \mod^\omega_{A_\bs} \;.
\]

\sssec{}
Given an analytic ring $\mathcal{B}$, we will use both $\mathcal{B}-\mod$ and $\mod_{\mathcal{B}}$ to denote the category
of $\mathcal{B}$-modules. There should usually be no room for confusion.

\sssec{}
To make our notation less cumbersome,
given a commutative algebra object $B \in \calg(\mod_{A_\bs})$ in the symmetric monoidal category $\mod_{A_\bs}$, we will
often use $B - \mod_{A_\bs}$ (as opposed to something like $B-\mod(\mod_{A_\bs})$) to denote the category of $B$-modules inside
$\mod_{A_\bs}$.

\sssec{} Given an analytic ring $\mathcal{B}$, we will denote the category of nuclear modules over $\mathcal{B}$ by $\nuc(\mathcal{B})$,
see Lecture 8 in \cite{complex} for the definition. Given a finite type scheme $X$ and a ring $R \in \calg(\qcoh(X_\bs))$,
we will denote by $\nuc(R, X_\bs)$ the category of nuclear objects of the category $R - \mod_{\qcoh(X_\bs)}$.

\ssec{Acknowledgments}
The author is grateful to Christopher Brav for conversations about the project.
Part of the research was done in the Centre of Pure Mathematics within MIPT, grant number FSMG-2023-0013.
The author benefited from a short stay at the Shanghai Institute for Mathematical and Interdisciplinary Sciences
while working on this paper and would like to thank them for hospitality.
\section{Preliminaries}

This section contains some technical preliminaries, mostly dealing with categories of nuclear modules over the adelic rings.

\ssec{}

Here we summarize part of the discussion from \cite{efimov2025localizinginvariantsinverselimits}, section 1.5.

\sssec{}\label{sssec-nuc-defn} Let $\CC$ be a compactly generated stable symmetric monoidal category whose monoidal unit is compact.
We now recall from \cite{complex}, Lecture 8, the definition of the full subcategory $\nuc(\CC) \subseteq \CC$ of nuclear objects.
Recall that a map $f \in \hom_\CC (x, y)$ is called trace-class if it lies in the image of the natural map \[
\hom_{\CC} (1_\CC, x^\vee \otimes_{\CC} y) \;\arr\; \hom_\CC(x, y) \;.
\] An object $M \in \CC$ is called nuclear if any map $f \colon P \arr M$ from a compact object $P \in \CC^\omega$ is trace-class.
The full subcategory $\nuc(\CC) \subseteq \CC$ is defined as consisting of nuclear objects. It follows from the definition that the
subcategory of nuclear objects is closed under colimits. Moreover, by Proposition 13.13 of \cite{analytic} and
Theorem 2.39 of \cite{ramzi2024dualizablepresentableinftycategories}, the category $\nuc(\CC)$ is $\omega_1$-presentable
and dualizable. Finally, in case the subcategory of compact objects $\CC^\omega \subset \CC$ is closed under the tensor product,
the category $\nuc(\CC)$ inherits the symmetric monoidal structure from $\CC$ (e.g. follows from Proposition 13.13 of \cite{analytic}).

\sssec{} We will also make a slight use of the functor of rigidification $(-)^\rig$
(see Construction 4.75 and Theorem 4.77 of \cite{ramzi2024locallyrigidinftycategories}),
which is the right adjoint to the fully faithful inclusion $\calg^\rig \hookrightarrow \calg(\pr^L_{\st, \kappa})$ for an uncountable $\kappa$.
It is easily seen (Propositions 1.23 and 1.27 of \cite{efimov2025localizinginvariantsinverselimits})
that, under the hypothesis of section \ref{sssec-nuc-defn}, the rigidification $\CC^\rig$ is a naturally a full subcategory in $\nuc(\CC)$.
The following proposition summarizes part of the discussion from \cite{efimov2025localizinginvariantsinverselimits}, section 1.5,
which will provide a description for the category of nuclear modules over certain analytic rings.

\begin{proposition}[\cite{efimov2025localizinginvariantsinverselimits}, Propositions 1.28-1.34; \cite{anschütz2024descentsolidquasicoherentsheaves}, Lemma 3.32]\label{prop-main-gen-nuc}
Let $\DD$ be a small symmetric monoidal stable category such that, for every $P \in \DD$, the object $P^{\op, \vee} \in \ind(\DD^\op)$
is nuclear. Then the following holds: \begin{enumerate}
\item the inclusion \[
\ind(\DD^\op)^\rig \;\subseteq\; \nuc(\ind(\DD^\op))
\] is an equivalence;
\item the inclusion \[
\ind(\DD)^\rig \;\subseteq\; \ind(\DD)
\] admits a continuous symmetric monoidal right adjoint, and the composite \[
\ind(\DD) \;\arr\; \ind(\DD)^\rig \;\simeq\; \ind(\DD^\op)^\rig \;\hookrightarrow\; \ind(\DD^\op)
\] maps an object $P \in \DD$ to $P^{\op, \vee} \in \ind(\DD^\op)$;
\item the inclusion \[
\ind(\DD^\op)^\rig \;\subseteq\; \ind(\DD^\op)
\] admits a continuous right adjoint.
\end{enumerate}
\end{proposition}

\begin{corollary}
Under the hypothesis of the previous proposition, the category $\nuc(\ind(\DD^\op))$ is generated under colimits by objects of the form
$P^{\op, \vee}$ with $P$ varying over the objects of $\DD$.
\end{corollary}

\ssec{}

Fix a commutative ring $A$ of finite type over $\Z$. Recall from \cite{brav2024beilinsonparshinadelessolidalgebraic},
section 1.4.5, that whenever we have a specialization closed subset $T \subseteq \spec(A)$, we define the functor \[
(-)^\thicksim_T \colon \mod_{A_\bs} \;\arr\; \mod_{A_\bs}
\] of completion at $T$ by the formula \[
(-)^\thicksim_T \;\stackrel{\define}{=}\; \uhom_{A_\bs} (\Ga_T, -) \;,
\] where \[\begin{split}
\Ga_T &\;\simeq\; \stackrel[\substack{\text{closed } Z \subseteq \spec(A) \\ Z \subseteq T}]{}{\colim} \Ga_Z \;, \\
\Ga_Z &\;\simeq\; \fib\bigl(A \to \Ga(\spec(A) \setminus Z, \O)\bigr) \;,
\end{split}\] is the local cohomology. By Corollary A.2.7 of \cite{brav2024beilinsonparshinadelessolidalgebraic}, the functor
of completion at $T$ is symmetric monoidal if restricted to the full subcategory \[
\mod^{< +\infty}_{A_\bs} \;\subseteq\; \mod_{A_\bs}
\] of eventually connective $A_\bs$-modules.

Likewise, recall from section 1.7 in \cite{brav2024beilinsonparshinadelessolidalgebraic} that,
for a generalization closed subset $E \subseteq \spec(A)$, we have the functor of completion at $E$
\[\begin{split}
&(-)^\thicksim_E \colon \;\mod_{A_\bs} \;\arr\; \mod_{A_\bs} \;,\\
&(-)^\thicksim_E \;=\; \uhom_{A_\bs}( A^\thicksim_{E^c}/A[-1], -) \;,
\end{split}\] where $E^c = \spec(A) \setminus E$ is the specialization closed subset of $\spec(A)$ given by the complement to $E$. 
It is isomorphic to the functor \[
- \otimes_{A_\bs} \Ga(E, \O)_\bullet \;,
\] where $\Ga(E, \O)_\bullet$ denotes the ultrasolid analytic ring structure on the idempotent $A$-algebra \[
\Ga(E, \O) \;\simeq\; \stackrel[\substack{{\text{open } U \subseteq \spec(A)}\\ E \subseteq U}]{}{\colim} \Ga(U, \O) \;,
\] see section 1.7.3 in \cite{brav2024beilinsonparshinadelessolidalgebraic}.

Now we would like to study nuclear modules over the solid commutative ring $A^\thicksim_{T,E} \in \calg(\mod_{A_\bs})$ ($A$ completed
at both $T$ and $E$).
We note that it is idempotent over $A_\bs$: \[\begin{split}
A^\thicksim_{T,E} \otimes_{A_\bs} A^\thicksim_{T,E} &\;\simeq\; \bigl( A^\thicksim_E \otimes_{A_\bs} A^\thicksim_E \bigr)^\thicksim_T \;\simeq\\
&\;\simeq\; \bigl( \Ga(E, \O) \otimes_{A_\bs} \Ga(E, \O) \bigr)^\thicksim_T \;\simeq\; \\
&\;\simeq\; \bigl( \Ga(E, \O) \bigr)^\thicksim_T \;=\; A^\thicksim_{T,E} \;,
\end{split}\] where we use Corollary A.2.7 of \cite{brav2024beilinsonparshinadelessolidalgebraic} to deal with the completion at $T$.
The next lemma claims that the hypothesis of Proposition \ref{prop-main-gen-nuc} holds for the category $A^\thicksim_{T,E}-\mod_{A_\bs}$.

\begin{lemma}
Fix a specialization closed subset $T \subseteq \spec(A)$ and a generalization closed subset $E \subseteq \spec(A)$.
The object \[
\uhom_{A_\bs} \Bigl(\prod_\N A, A^\thicksim_{T,E} \Bigr) \;\simeq\; \Bigl(\bigoplus_\N A \Bigr)^\thicksim_{T,E} \;\in\; A^\thicksim_{T,E} - \mod_{A_\bs}
\] is nuclear.
\end{lemma}
\begin{proof}
It suffices to prove that the map \[
\Bigl(\bigoplus_\N A \Bigr)^\thicksim_{T,E} \otimes_{A_\bs} \Bigl(\bigoplus_\N A \Bigr)^\thicksim_{T,E} \;\arr\; \Bigl(\bigoplus_{\N\times\N} A \Bigr)^\thicksim_{T,E}
\] is an isomorphism, which immediately follows from Corollary A.2.7, \cite{brav2024beilinsonparshinadelessolidalgebraic},
combined with the fact that the functor of completion at $E$ is continuous.
\end{proof}

The same result can be obtained for the same ring $A^\thicksim_{T,E}$, but now equipped with the ultrasolid analytic ring structure,
which in this case can be constructed as the push-out \[
A^\thicksim_{T,E \: \bullet} \;\simeq\; A^\thicksim_{E \:\bullet} \otimes_{A_\bs} (A^\thicksim_{T}, A_\bs) \;.
\] Indeed, using the base change, we get \[\begin{split}
\Bigl( A^\thicksim_{E \:\bullet} \otimes_{A_\bs} (A^\thicksim_{T}, A_\bs) \Bigr) \otimes_{A_\bs} \prod_\N A &\;\simeq\; A^\thicksim_{E, \:\bullet} \otimes_{A_\bs} \Bigl( A^\thicksim_T \otimes_{A_\bs} \prod_\N A \Bigr) \;\simeq\; \\
&\;\simeq\; A^\thicksim_{E, \:\bullet} \otimes_{A_\bs} \Bigl( \prod_\N A^\thicksim_T \Bigr) \;\simeq\; \\
&\;\simeq\; \Bigl( \prod_\N A^\thicksim_T \Bigr)^\thicksim_E \;\simeq\; \\
&\;\simeq\; \prod_\N A^\thicksim_{T,E} \;.
\end{split}\]

\begin{lemma}
The object \[
\uhom_{A^\thicksim_{T,E \: \bullet}} \Bigl(\prod_\N A^\thicksim_{T,E}, A^\thicksim_{T,E} \Bigr) \;\simeq\; \Bigl(\bigoplus_\N A \Bigr)^\thicksim_{T,E} \;\in\; A^\thicksim_{T,E \:\bullet} - \mod
\] is nuclear.
\end{lemma}
\begin{proof}
The same argument applies.
\end{proof}

The following two corollaries now follow directly from Proposition \ref{prop-main-gen-nuc}. 
\begin{corollary}\label{cor-nuc-of-a-completion}
\begin{enumerate}
\item The inclusion \[
\Bigl(A^\thicksim_{T,E} - \mod_{A_\bs} \Bigr)^\rig \;\subseteq\; \nuc(A^\thicksim_{T,E}, A_\bs)
\] is an equivalence. The full subcategory \[
\nuc(A^\thicksim_{T,E}, A_\bs) \;\subseteq\; A^\thicksim_{T,E} - \mod_{A_\bs}
\] is generated under colimits by the object \[
\Bigl(\bigoplus_\N A \Bigr)^\thicksim_{T,E} \;.
\]
\item The inclusion \[
\Bigl(A^\thicksim_{T,E \:\bullet} - \mod \Bigr)^\rig \;\subseteq\; \nuc(A^\thicksim_{T,E \:\bullet})
\] is an equivalence. The full subcategory \[
\nuc(A^\thicksim_{T,E \:\bullet}) \;\subseteq\; A^\thicksim_{T,E \:\bullet} - \mod
\] is generated under colimits by the object \[
\Bigl(\bigoplus_\N A \Bigr)^\thicksim_{T,E} \;.
\]
\end{enumerate}
In particular, the functor \[
A^\thicksim_{T,E} - \mod_{A_\bs} \;\arr\; A^\thicksim_{T,E \:\bullet} - \mod
\] induces an equivalence \[
\nuc(A^\thicksim_{T,E}, A_\bs) \;\simeq\; \nuc(A^\thicksim_{T,E \:\bullet}) \;.
\]
\end{corollary}

\begin{corollary}\label{cor-nuc-is-retract}
The inclusion \[
\nuc(A^\thicksim_{T,E}, A_\bs) \;\hookrightarrow\; A^\thicksim_{T,E} - \mod_{A_\bs}
\] admits a continuous $\nuc(A^\thicksim_{T,E}, A_\bs)$-linear right adjoint.
\end{corollary}
\begin{proof}
The existence and continuity follow immediately from the third claim of Proposition \ref{prop-main-gen-nuc}.
Because the inclusion is symmetric monoidal, the right adjoint is lax $\nuc(A^\thicksim_{T,E}, A_\bs)$-linear,
but any lax $\nuc(A^\thicksim_{T,E}, A_\bs)$-linear functor must be strict by the rigidity of $\nuc(A^\thicksim_{T,E}, A_\bs)$,
see \cite{GR1}, Lemma 9.3.6.
\end{proof}

\begin{corollary}\label{cor-nuc-of-comp-is-idempotent}
The algebra \[
\nuc(A^\thicksim_{T,E}, A_\bs) \;\in\; \calg(\pr^{L}_\st)
\] is idempotent over $\nuc(A_\bs)$.
\end{corollary}
\begin{proof}
By Corollary \ref{cor-nuc-is-retract}, the category $\nuc(A^\thicksim_{T,E}, A_\bs)$
is a retract of $A^\thicksim_{T,E} - \mod_{A_\bs}$ in $\pr^L_\st$. Likewise, the category $\nuc(A_\bs)$ is a retract
of $\mod_{A_\bs}$ in $\pr^L_\st$. Moreover, by the linearity of the right adjoints, the tensor product \[
\nuc(A^\thicksim_{T,E}, A_\bs) \otimes_{\nuc(A_\bs)} \nuc(A^\thicksim_{T,E}, A_\bs)
\] is also a retract of \[
A^\thicksim_{T,E} - \mod_{A_\bs} \otimes_{\mod_{A_\bs}} A^\thicksim_{T,E} - \mod_{A_\bs} \;\simeq\; \bigl( A^\thicksim_{T,E} \otimes_{A_\bs} A^\thicksim_{T,E} \bigr) - \mod_{A_\bs} \;\simeq\; A^\thicksim_{T,E} - \mod_{A_\bs} \;,
\] where we use that the algebra $A^\thicksim_{T,E}$ is idempotent over $A_\bs$.
Then it follows that the multiplication map \[
\nuc(A^\thicksim_{T,E}, A_\bs) \otimes_{\nuc(A_\bs)} \nuc(A^\thicksim_{T,E}, A_\bs) \;\arr\; \nuc(A^\thicksim_{T,E}, A_\bs)
\] is fully faithful. But it is also essentially surjective for obvious reasons, and the claim follows.
\end{proof}

\ssec{Nuclear modules over the adelic rings}
In this section, we record some basic technical results about nuclear modules over the adelic rings.
Let $X$ be a scheme of finite type over $\Z$. Let $d$ denote its dimension.
Recall from \cite{brav2024beilinsonparshinadelessolidalgebraic},
section 3.2, the skeletal filtration \[
\emptyset = S_{-1} \subseteq S_0 \subseteq \ldots \subseteq S_d = X \;,
\] which is a filtration on $X$ by specialization closed subsets \[
S_k \;=\; \{ \p \in X : \dim \overline{\{\p\}} \le k \} \;\subseteq\; X \;.
\] Associated with the skeletal filtration, we have the solid adelic rings \[
L_k \O_X \;=\; \bigl(\O_X\bigr)^\thicksim_{S_k, S^c_{k-1}} \;\simeq\; \prod_{\substack{\p \in X \\ \dim \overline{\{\p\}} = k}} \leftcomp\: \O_{X, \p}
\] and \[\begin{split}
L_{k_0} \ldots L_{k_m} \O_X &\;=\; \Bigl( \bigl(\O_X\bigr)^\thicksim_{S_{k_0}} \otimes_{X_\bs} L_{k_1} \ldots L_{k_m} \O_X \Bigr)^\thicksim_{S^c_{k_0 - 1}} \;\simeq\; \\
&\empty \\
&\;\simeq\; \prod_{\substack{\p_m \in X \\ \dim \overline{\{\p_m\}} = k_m}} \leftcomp\: (\ldots \prod_{\substack{\p_1\in \overline{\{\p_2\}}\\\dim \overline{\{\p_1\}} = k_1}}\leftcomp\:(\prod_{\substack{\p_0 \in \overline{\{\p_1\}} \\ \dim \overline{\{\p_0\}} = k_0}} \leftcomp\:\O_{X,\p_0})_{\p_1} \ldots )_{\p_m} \;\simeq\; \\
&\empty \\
&\;\simeq\; L_{k_0} \O_X \otimes_{X_\bs} \ldots \otimes_{X_\bs} L_{k_m} \O_X \;,
\end{split}\] which all are idempotent commutative algebras in $\qcoh(X_\bs)$, see \cite{brav2024beilinsonparshinadelessolidalgebraic},
Theorem 3.2.1, Proposition 4.2.1, and Corollary 4.2.2.

\sssec{}\label{sssec-adelic-ring-restricts}
For a Zariski open embedding $j \colon U \hookrightarrow X$, we have a symmetric monoidal restriction functor \[
j^* \colon \qcoh(X_\bs) \;\arr\; \qcoh(U_\bs) \;,
\] which preserves all limits and colimits, and in particular, maps an adelic ring of $X$ into the corresponding adelic ring of $U$: \[
j^* L_{k_0} \ldots L_{k_m} \O_X \;\simeq\; \prod_{\substack{\p_m \in U \\ \dim \overline{\{\p_m\}} = k_m}} \leftcomp\: (\ldots \prod_{\substack{\p_1\in \overline{\{\p_2\}} \cap U \\\dim \overline{\{\p_1\}} = k_1}}\leftcomp\:(\prod_{\substack{\p_0 \in \overline{\{\p_1\}} \cap U \\ \dim \overline{\{\p_0\}} = k_0}} \leftcomp\:\O_{U,\p_0})_{\p_1} \ldots )_{\p_m} \;.
\] We are going to use that observation in the proof of the following proposition, which checks that the category \[
L_{k_0} \ldots L_{k_m} \O_X - \mod_{\qcoh(X_\bs)}
\] satisfies the hypothesis of Proposition \ref{prop-main-gen-nuc}.

\begin{proposition}\label{prop-nuc-adelic}
Let $X$ be a scheme of finite type over $\Z$. For any compact object \[
P \;\in\; \qcoh(X_\bs)^\omega \;, 
\] the dual \[
\uhom_{X_\bs} \bigl( P, L_{k_0}\ldots L_{k_m} \O_X )
\] is a nuclear $L_{k_0}\ldots L_{k_m}\O_X$-module.
\end{proposition}

We split the proof into several parts.

\begin{lemma}\label{lem-nuc-adelic}
For a tuple $0 \le k_0 < k_1 < \ldots < k_m \le d$, the natural map \begin{equation}\label{nuc-over-adelic-eq1}
\uhom_{X_\bs} \bigl( P, L_{k_0}\O_X \bigr) \otimes_{X_\bs} L_{k_1}\O_X \otimes_{X_\bs} \ldots \otimes_{X_\bs} L_{k_m}\O_X \;\arr\; \uhom_{X_\bs} \bigl( P, L_{k_0}\ldots L_{k_m} \O_X )
\end{equation}
is an isomorphism.
\end{lemma}
\begin{proof}
By Zariski descent for solid quasi-coherent sheaves, see Theorem 9.8 of \cite{condensed}, it suffices to prove that
the map (\ref{nuc-over-adelic-eq1}) becomes an isomorphism after restriction to any affine open $j \colon \spec(A) \hookrightarrow X$.
Recall that the functor \[
j^* \colon \qcoh(X_\bs) \;\arr\; \mod_{A_\bs}
\] admits a $\qcoh(X_\bs)$-linear fully faithful left adjoint $j_!$. In particular, the restriction functor respects the inner $\hom$:
\[
j^* \uhom_{X_\bs} (M, N) \;\simeq\; \uhom_{U_\bs}(j^*M, j^*N) \;.
\]

Because the adelic rings are respected by the restriction functor $j^*$, we have reduced the proof to the affine case $X = \spec(A)$.
We are left to prove that the map \[
\uhom_{A_\bs} \bigl( \prod_{\N} A , L_{k_0}A \bigr) \otimes_{A_\bs} L_{k_1}A \otimes_{A_\bs} \ldots \otimes_{A_\bs} L_{k_m}A \;\arr\; \uhom_{A_\bs} \bigl( \prod_\N A, L_{k_0}\ldots L_{k_m} A \bigr)
\] is an isomorphism. To make the formulas look less cumbersome, we only treat the case $m = 1$. The general case is proved similarly
using induction on $m$.

We now prove that the map \begin{equation}\label{nuc-over-adelic-eq2}
\xymatrix{
\uhom_{A_\bs} \bigl( \prod_{\N} A , L_{k_0}A \bigr) \otimes_{A_\bs} L_{k_1}A \ar[rr] && \uhom_{A_\bs} \bigl( \prod_\N A, L_{k_0} L_{k_1} A \bigr) \\
\prod_{\substack{\p \in \spec(A) \\ \dim(A/\p) = k_0}} \leftcomp\:\bigl( \oplus_\N A \bigr)_\p \otimes_{A_\bs} \prod_{\substack{\q \in \spec(A) \\ \dim (A/\q) = k_1}} \leftcomp A_\q \ar[rr]\ar@{}[u]|\invsimeq && \prod_{\substack{\q \in \spec(A) \\ \dim(A/\q) = k_1}} \leftcomp\: \Bigl( \prod_{\substack{\p \supset \q \\ \dim(A/\p) = k_0}} \leftcomp\: \bigl( \oplus_\N A \bigr)_\p \Bigr)_\q  \ar@{}[u]|\invsimeq
}
\end{equation} is an isomorphism. By Proposition A.2.1, \cite{brav2024beilinsonparshinadelessolidalgebraic},
both the source and the target of (\ref{nuc-over-adelic-eq2}) are $S_{k_1}$-complete.
Therefore, by Proposition A.1.1, \cite{brav2024beilinsonparshinadelessolidalgebraic}, it suffices to prove that the map
(\ref{nuc-over-adelic-eq2}) becomes an isomorphism upon application of $- \otimes_{A_\bs} (A/\q)_\bs$ for any point $\q \in \spec(A)$
of dimension not greater than $k_1$. In case dimension of the point $\q$ is strictly less than $k_1$, both the source and the 
target vanish. In case the dimension is exactly $k_1$, we use Lemma 1.1.7 and Lemma 1.3.1, 
\cite{brav2024beilinsonparshinadelessolidalgebraic}, to distribute
the tensor product with the products and completions, and the map turns into \[
A_\q / \q \otimes_{{A/\q}_\bs} \prod_{\substack{\p \supset \q \\ \dim(A/\p) = k_0}} \leftcomp\:\bigl( \oplus_\N A/\q \bigr)_\p \;\stackrel{=}{\arr} A_\q / \q \otimes_{{A/\q}_\bs} \prod_{\substack{\p \supset \q \\ \dim(A/\p) = k_0}} \leftcomp\:\bigl( \oplus_\N A/\q \bigr)_\p \;.
\]
\end{proof}

\begin{proof}[Proof of Proposition \ref{prop-nuc-adelic}]
By Lemma \ref{lem-nuc-adelic}, it suffices to treat the case $m = 0$ and $m=1$. Regarding the former case,
it is well known that the subcategory $\nuc(X_\bs) \subset \qcoh(X_\bs)$ of nuclear sheaves is equivalent to the full subcategory
$\qcoh(X) \subset \qcoh(X_\bs)$ of classical (or discrete) quasi-coherent sheaves, and the statement says that, for a compact
object $P \in \qcoh(X_\bs)^\omega$, the dual $\uhom_{X_\bs}(P, \O_X)$ is discrete. The latter is clear.

In the case $m=1$, it suffices to prove that, for a pair $P_1, P_2 \in \qcoh(X_\bs)^\omega$ of compact objects,
the natural map \[
\uhom_{X_\bs}(P_1, L_k\O_X) \otimes_{X_\bs} \uhom_{X_\bs}(P_2, L_k\O_X) \;\arr\; \uhom_{X_\bs} (P_1 \otimes_{X_\bs} P_2, L_k\O_X )
\] is an isomorphism. As in the proof of Lemma \ref{lem-nuc-adelic}, it suffices to treat the affine case $X = \spec(A)$,
where it suffices to prove that the map \[
\prod_{\substack{\p \in \spec(A) \\ \dim(A/\p) = k}} \leftcomp\: \bigl( \oplus_\N A \bigr)_\p \otimes_{A_\bs} \prod_{\substack{\p \in \spec(A) \\ \dim(A/\p) = k}} \leftcomp\: \bigl( \oplus_\N A \bigr)_\p \;\arr\; \prod_{\substack{\p \in \spec(A) \\ \dim(A/\p) = k}} \leftcomp\: \bigl( \oplus_{\N\times\N} A \bigr)_\p
\] is an isomorphism. That immediately follows from Proposition A.2.1, \cite{brav2024beilinsonparshinadelessolidalgebraic}.
\end{proof}

\begin{corollary}\label{cor-nuc-over-adelic-and-rig}
The inclusion \[
\Bigl( L_{k_0}\ldots L_{k_m}\O_X - \mod_{\qcoh(X_\bs)} \Bigr)^\rig \;\subseteq\; \nuc(L_{k_0}\ldots L_{k_m}\O_X, X_\bs)
\] is an equivalence; the full subcategory \begin{equation}\label{cor-nuc-of-adelic-eq1}
\nuc(L_{k_0}\ldots L_{k_m}\O_X, X_\bs) \;\subseteq\; L_{k_0}\ldots L_{k_m}\O_X - \mod(\qcoh(X_\bs))
\end{equation}
is generated under colimits by objects of the form \[
\uhom_{X_\bs} \bigl( P, L_{k_0}\ldots L_{k_m}\O_X \bigr)
\] where $P$ varies over the compact objects of $\qcoh(X_\bs)$;
the inclusion (\ref{cor-nuc-of-adelic-eq1}) admits a continuous $\nuc(L_{k_0}\ldots L_{k_m}\O_X, X_\bs)$-linear right adjoint.
\end{corollary}
\begin{proof}
Immediately follows from Proposition \ref{prop-main-gen-nuc} and Proposition \ref{prop-nuc-adelic}. Linearity of the right adjoint
is proved as in the proof of Corollary \ref{cor-nuc-is-retract}.
\end{proof}

Before we state the next corollary, recall from Proposition 4.2.1, \cite{brav2024beilinsonparshinadelessolidalgebraic},
that, for a tuple $0 \le k_0 < k_1 < \ldots < k_m \le d$, the map \[
L_{k_0} \O_X \otimes_{X_\bs} \ldots \otimes_{X_\bs} L_{k_m} \O_X \;\arr\; L_{k_0}\ldots L_{k_m} \O_X
\] is an isomorphism of commutative algebras in $\qcoh(X_\bs)$. In particular, the multiplication map \[
L_{k_0}\O_X - \mod_{\qcoh(X_\bs)} \otimes_{\qcoh(X_\bs)} \ldots \otimes_{\qcoh(X_\bs)} L_{k_m}\O_X - \mod_{\qcoh(X_\bs)} \;\arr\; L_{k_0}\ldots L_{k_m}\O_X - \mod_{\qcoh(X_\bs)}
\] is an equivalence. The next corollary shows that the same holds for nuclear modules.
\begin{corollary}\label{cor-adelic-nuc-splits-as-tensor-product}
The multiplication functor \[
\nuc(L_{k_0}\O_X, X_\bs) \otimes_{\nuc(X_\bs)} \ldots \otimes_{\nuc(X_\bs)} \nuc(L_{k_m}\O_X, X_\bs) \;\arr\; \nuc(L_{k_0}\ldots L_{k_m}\O_X, X_\bs)
\] is an equivalence.
\end{corollary}
\begin{proof}
First, we note that the functor in question is fully faithful, because, by Corollary \ref{cor-nuc-over-adelic-and-rig},
the category \[
\nuc(L_{k_0}\O_X, X_\bs) \otimes_{\nuc(X_\bs)} \ldots \otimes_{\nuc(X_\bs)} \nuc(L_{k_m}\O_X, X_\bs)
\] is a retract of \[
L_{k_0}\O_X - \mod_{\qcoh(X_\bs)} \otimes_{\qcoh(X_\bs)} \ldots \otimes_{\qcoh(X_\bs)} L_{k_m}\O_X - \mod_{\qcoh(X_\bs)}
\] in $\pr^L_\st$. Thus, it suffices to prove that our functor is essentially surjective.
By Corollary \ref{cor-nuc-over-adelic-and-rig}, it suffices to prove that each object of the form \[
\uhom_{X_\bs} (P, L_{k_0} \ldots L_{k_m} \O_X) \;,
\] where $P \in \qcoh(X_\bs)$ is compact, is in the image. But that immediately follows from Lemma \ref{lem-nuc-adelic}.
\end{proof}

\sssec{} Recall from section \ref{sssec-adelic-ring-restricts} that adelic rings behave well with respect to the restriction
to a Zariski open subset. Here we expand on that observation and record the behavior of the category of nuclear modules over an
adelic ring with respect to the restriction to a Zariski open subset.

\begin{proposition}\label{prop-upper-star-for-Zar-is-LL-localization}
Let $j \colon U \hookrightarrow X$ be a Zariski open embedding of schemes of finite type over $\Z$. The restriction functor
\[
j^* \colon \nuc(L_{k_0}\ldots L_{k_m} \O_X, X_\bs) \;\arr\; \nuc(j^* L_{k_0}\ldots L_{k_m} \O_X, U_\bs)
\] admits a fully faithful continuous right adjoint.
\end{proposition}
\begin{proof}
It suffices to prove that the lower-$*$ push-forward \[
j_* \colon j^*L_{k_0}\ldots L_{k_m} \O_X - \mod_{\qcoh(U_\bs)} \;\hookrightarrow\; L_{k_0}\ldots L_{k_m} \O_X - \mod_{\qcoh(X_\bs)}
\] preserves nuclear objects. Because the restriction $j^* \colon \qcoh(X_\bs) \;\arr\; \qcoh(U_\bs)$ is essentially surjective on
compact objects, it suffices to prove that, for a compact object $P\in \qcoh(X_\bs)^\omega$,
the $L_{k_0} \ldots L_{k_m} \O_X$-module \[
j_*\uhom_{U_\bs} (j^*P, j^*L_{k_0} \ldots L_{k_m} \O_X) \;\simeq\; \uhom_{X_\bs}(P, j_*j^* L_{k_0} \ldots L_{k_m} \O_X)
\] is nuclear. That follows from Proposition \ref{prop-nuc-adelic}
because $j_* j^*L_{k_0} \ldots L_{k_m} \O_X$ is a retract of $L_{k_0} \ldots L_{k_m} \O_X$, which in turn follows from the formula
describing the adelic rings as products over flags of points in $X$---see section \ref{sssec-adelic-ring-restricts}.
\end{proof}

\section{The main theorem}

Let $\E$ be an accessible stable category.
Fix an accessible localizing invariant \[
F \colon \pr_\st^\dual \;\arr\; \E \;.
\]

\begin{theorem}\label{thm-nuc-adelic-descent}
The canonical map \[
F\bigl(\nuc(X_\bs)\bigr) \;\arr\; \stackrel[\substack{0 \le k_0 < k_1 < \ldots < k_m \le d \\ 0\le m}]{}{\lim} F\bigl(\nuc(L_{k_0}\ldots L_{k_m} \O_X, X_\bs )\bigr)
\] is an isomorphism.
\end{theorem}
The strategy of the proof is similar to that of \cite{kim2021adelicdescentktheory},
and is basically to split to cubical diagram into fiber sequences, using the cubical reduction principle,
see Proposition 2.2.2 in \cite{kim2021adelicdescentktheory}.

\ssec{Reduction to the affine case}\label{ssec-reduction-to-affine}
We first reduce the proof of Theorem \ref{thm-nuc-adelic-descent} to the case
where the scheme $X$ is affine. Assume we have a decomposition \[
X \;=\; U\;\cup\; V
\] of $X$ into a union of a pair of Zariski open subschemes. By Zariski descent for solid quasi-coherent sheaves,
Theorem 9.8 of \cite{condensed}, the square \[\xymatrix{
L_{k_0}\ldots L_{k_m} \O_X - \mod_{\qcoh(X_\bs)} \ar[rr]\ar[d] && (L_{k_0}\ldots L_{k_m} \O_X)|_{V_\bs} - \mod_{\qcoh(V_\bs)} \ar[d] \\
(L_{k_0}\ldots L_{k_m} \O_X)|_{U_\bs} - \mod_{\qcoh(U_\bs)} \ar[rr] && (L_{k_0}\ldots L_{k_m} \O_X)|_{(U\cap V)_\bs} - \mod_{\qcoh((U\cap V)_\bs)}
}\] is a pull-back square in $\calg(\pr^L_\st)$. By passing to the nuclear objects and using
Proposition \ref{prop-upper-star-for-Zar-is-LL-localization}, we get that the square
\begin{equation}\label{nuc-pull-back-Zar-open-diag}
\xymatrix{
\nuc(L_{k_0}\ldots L_{k_m} \O_X, X_\bs) \ar[rr] \ar[d] && \nuc((L_{k_0}\ldots L_{k_m} \O_X)|_{V_\bs}, V_\bs) \ar[d] \\
\nuc((L_{k_0}\ldots L_{k_m} \O_X)|_{U_\bs}, U_\bs) \ar[rr] && \nuc((L_{k_0}\ldots L_{k_m} \O_X)|_{(U\cap V)_\bs}, (U\cap V)_\bs)
}\end{equation}
is a pull-back square in $\calg(\pr^L_\st)$. Namely, there is a functor from the category
$\nuc(L_{k_0}\ldots L_{k_m} \O_X, X_\bs)$ to the pull-back, and it is fully faithful by construction. The essential surjectivity
can be proved directly using Proposition \ref{prop-upper-star-for-Zar-is-LL-localization} and the base change
for solid quasi-coherent sheaves with respect to Zariski open embeddings, see Example 13.15 in \cite{analytic}.

We get that the induced map on kernels of the horizontal arrows in the diagram (\ref{nuc-pull-back-Zar-open-diag}) is an
equivalence. In addition, by Proposition \ref{prop-upper-star-for-Zar-is-LL-localization},
each functor in the square is a strongly continuous localization, and it follows that the square is mapped to a pull-back
square by any stable localizing invariant.
We conclude that, if the claim of Theorem \ref{thm-nuc-adelic-descent} holds for both $U$ and $V$, and for their intersection $U\cap V$,
then it also holds for $X$. Using this observation, we can run a standard induction on the size of an affine covering
reducing to the case of a separated $X$ first, and to the case of an affine $X$ after.

\ssec{The affine case}
We now assume that $X = \spec(A)$ for a commutative algebra $A$ of finite type over $\Z$. Denote its dimension by $d$.
For any $k \in 0\ldots d$, associated with the $k$-th skeleton $S_k \subseteq \spec(A)$,
we have a recollement \[\xymatrix{
A^\thicksim_{S_k}-\mod_{A_\bs} \ar@{^{(}->}[rr] && \mod_{A_\bs} \ar@<2ex>[ll]\ar@<-2ex>[ll] \ar[rr] && \Ga(S_k^c, \O)_\bullet - \mod \ar@<2ex>@{_{(}->}[ll]\ar@<-2ex>@{_{(}->}[ll] & .
}\]
We define a full subcategory \[
\nuc_{\le k} (A_\bs) \;\subseteq\; \nuc(A_\bs)
\] as the kernel \[\begin{split}
\nuc_{\le k} (A_\bs) \;\stackrel{\define}{=}&\; \mathrm{Ker} \Bigl( \nuc(A_\bs) \;\arr\; \nuc(\Ga(S_k^c, \O)_\bullet) \Bigr) \;. \\
\end{split}\]
We note that the ring $\Ga(S_k^c, \O)$ is discrete, and therefore,
the right adjoint to the induction functor along the map $A_\bs \arr \Ga(S_k^c, \O)_\bullet$ preserves nuclear objects.

We thus get a recollement \[\xymatrix{
\nuc_{\le k} (A_\bs) \ar@{^{(}->}[rr] && \nuc(A_\bs) \ar@<2ex>[ll]\ar@{}@<-2ex>[ll] \ar[rr] && \nuc(\Ga(S_k^c, \O)_\bullet) \ar@<2ex>@{_{(}->}[ll]\ar@{}@<-2ex>[ll] \\
&& M \ar@{}[u]|\invertin \ar@{|->}[rr] && \Ga(S_k^c, \O) \otimes_A M \ar@{}[u]|\invertin &.
}\] In particular, the full subcategory \[
\nuc_{\le k} (A_\bs) \;\subseteq\; \nuc(A_\bs) \;\simeq\; \mod_A
\] consists of those $A$-modules which are $S_k$-torsion. In other words, \[
\nuc_{\le k} (A_\bs) \;=\; \{ M \in \mod_{A} \;:\; M\otimes_A \Ga_{S_k} \;\stackrel{\simeq}{\arr}\; M \} ,
\] where \[
\Ga_{S_k} \;\simeq\; \stackrel[\substack{\text{closed } Z \subseteq \spec(A) \\ \dim(Z) \le k}]{}{\colim} \Ga_Z
\] denotes the local cohomology with supports on $S_k$.

\sssec{}
Because the map $A_\bs \;\arr\; \Ga(S_{k+1}^c, \O)_\bullet$ factors over the map $A_\bs \;\arr\; \Ga(S_{k}^c, \O)_\bullet$,
we get an inclusion \begin{equation}\label{idk-eq}
\nuc_{\le k}(A_\bs) \;\subseteq\; \nuc_{\le k+1} (A_\bs) \;.
\end{equation}
We also note that the full subcategory \[
\nuc_{\le k}(A_\bs) \;\subseteq\; \nuc(A_\bs)
\] is naturally a tensor ideal, and the inclusion (\ref{idk-eq}) is $\nuc(A_\bs)$-linear.

\sssec{}
Analyzing the full subcategories $\nuc_{\le k}(A_\bs) \;\subseteq\; \nuc(A_\bs)$ will allow us to split the cubical diagram
of Theorem \ref{thm-nuc-adelic-descent} into fiber sequences,
which will in turn allow us to prove the theorem. Specifically, the following
proposition is key to that splitting.
\begin{proposition}\label{prop-main-tech-for-induction-step}
\begin{enumerate}
\item The action \[
\nuc(A_\bs) \;\curvearrowright\; \nuc_{\le k} (A_\bs)
\] factors through the map \[
\nuc(A_\bs) \;\arr\; \nuc(A^\thicksim_{S_k}, A_\bs) \;.
\] We note that, by Corollary \ref{cor-nuc-of-comp-is-idempotent}, such factoring is a property rather then a structure.
\item The functor \[
\nuc_{\le k} (A_\bs) \;\arr\; \nuc_{\le k}(A_\bs) \otimes_{\nuc(A_\bs)} \nuc(L_kA, A_\bs)
\]
admits a fully faithful continuous $\nuc(A_\bs)$-linear right adjoint.
\item The natural map \[\begin{split}
\nuc_{\le k-1}(A_\bs) \;\arr\; \mathrm{Ker} &\Bigl( \nuc_{\le k} (A_\bs) \;\arr\; \nuc_{\le k}(A_\bs) \otimes_{\nuc(A_\bs)} \nuc(L_kA, A_\bs) \Bigr) \;.
\end{split}\] is an equivalence.
\end{enumerate}
\end{proposition}

\sssec{Proof of Theorem \ref{thm-nuc-adelic-descent}}
We first use Proposition \ref{prop-main-tech-for-induction-step}
to finish the proof of the theorem, and then supply a proof of the proposition.
We prove by induction on $k$ that,
for any $k \in 0..d$, the natural map \[
F\Bigl(\nuc_{\le k} (A_\bs) \Bigr) \;\arr\; \stackrel[\substack{0 \le k_0 < \ldots < k_m \le k \\ m \ge 0}]{}{\lim} F\Bigl(\nuc_{\le k}(A_\bs) \otimes_{\nuc(A_\bs)} \nuc(L_{k_0}A, A_\bs) \otimes_{\nuc(A_\bs)} \ldots \otimes_{\nuc(A_\bs)} \nuc(L_{k_m}A, A_\bs) \Bigr)
\] is an isomorphism. Note that the claim of the theorem is a particular case of this statement since,
by Corollary \ref{cor-adelic-nuc-splits-as-tensor-product}, we have an equivalence \[
\nuc(L_{k_0}A, A_\bs) \otimes_{\nuc(A_\bs)} \ldots \otimes_{\nuc(A_\bs)} \nuc(L_{k_m}A, A_\bs) \;\simeq\; \nuc(L_{k_0} \ldots L_{k_m} A, A_\bs) \;.
\]

\smallskip
\smallskip
\smallskip
\noindent \underline{Base: $k = 0$}. It suffices to prove that the functor \[
\nuc_{\le 0} (A_\bs) \;\arr\; \nuc_{\le 0} (A_\bs) \otimes_{\nuc(A_\bs)} \nuc(L_0A, A_\bs)
\] is an equivalence. That immediately follows from the first claim of Proposition \ref{prop-main-tech-for-induction-step} in view of
the isomorphism $L_0A \;\simeq\; A^\thicksim_{S_0}$.

\smallskip
\smallskip
\smallskip
\noindent \underline{The induction step}. We assume that the claim holds for some $k \ge 0$.
By the second claim of Proposition \ref{prop-main-tech-for-induction-step}, the functor
\begin{equation}\label{thm-ind-step-eq0}
\nuc_{\le k+1} (A_\bs) \;\arr\; \nuc_{\le k+1}(A_\bs) \otimes_{\nuc(A_\bs)} \nuc(L_{k+1}A, A_\bs)
\end{equation}
is a strongly continuous $\nuc(A_\bs)$-linear localization, and, in particular, it remains a strongly continuous localization
after tensoring with
\[
\nuc(L_{k_0}A, A_\bs) \otimes_{\nuc(A_\bs)} \ldots \otimes_{\nuc(A_\bs)} \nuc(L_{k_m}A, A_\bs)
\]
for any tuple $0 \le k_0 < \ldots < k_m \le k$.

Let us denote by $\K$ the kernel of the functor (\ref{thm-ind-step-eq0}).
By the cubical reduction
principle, see Proposition 2.2.2 in \cite{kim2021adelicdescentktheory},
it suffices to prove that the map \[
F\Bigl(\K \Bigr) \;\arr\; \stackrel[\substack{0 \le k_0 < \ldots < k_m \le k \\ m \ge 0}]{}{\lim} F\Bigl(\K \otimes_{\nuc(A_\bs)} \nuc(L_{k_0}A, A_\bs) \otimes_{\nuc(A_\bs)} \ldots \otimes_{\nuc(A_\bs)} \nuc(L_{k_m}A, A_\bs) \Bigr)
\] is an isomorphism. But that is exactly the induction hypothesis because, by the last claim of
Proposition \ref{prop-main-tech-for-induction-step}, we have a $\nuc(A_\bs)$-linear equivalence \[
\K \;\simeq\; \nuc_{\le k} (A_\bs) \;.
\] Theorem \ref{thm-nuc-adelic-descent} is now proved.

\ssec{Proof of Proposition \ref{prop-main-tech-for-induction-step}}

Much like the proposition itself, the proof consists of several parts, which are recorded as separate statements.
\begin{lemma}\label{lem-main-prop-part1}
The action \[
\nuc(A_\bs) \;\curvearrowright\; \nuc_{\le k} (A_\bs)
\] factors through the map \[
\nuc(A_\bs) \;\arr\; \nuc(A^\thicksim_{S_k}, A_\bs) \;.
\] We note that, by Corollary \ref{cor-nuc-of-comp-is-idempotent}, such factoring is a property rather then a structure.
\end{lemma}

\begin{proof}
We prove that the full subcategory \[
\nuc_{\le k} (A_\bs) \;\subset\; A^\thicksim_{S_k}-\mod_{A_\bs}
\] is stable under the action of $\nuc(A^\thicksim_{S_k}, A_\bs)$.
We note that the full subcategory \[
\nuc_{\le k}(A_\bs) \;\subseteq\; \nuc(A_\bs) \;\simeq\; \mod_A
\] consists of $S_k$-torsion modules, i.e. those $M \in \mod_A$ satisfying $\Ga_{S_k} \otimes_A M \stackrel{\simeq}{\arr} M$.
In particular, the category $\nuc_{\le k}(A_\bs)$ is generated under colimits by objects of the form $A/I$
where $I \subset A$ is an ideal such that $\dim(A/I) \le k$. Also, by Corollary \ref{cor-nuc-of-a-completion},
the category $\nuc(A^\thicksim_{S_k}, A_\bs)$ is generated under colimits
by the object $(\oplus_\N A)^\thicksim_{S_k}$.
It suffices to prove that the tensor product \[
A/I \otimes_{A_\bs} \bigl( \bigoplus_\N A \bigr)^\thicksim_{S_k}
\] lies inside the subcategory $\nuc_{\le k}(A_\bs)$, which is shown by the following computation, analogous to the computation
from the proof of Proposition A.1.4 of \cite{brav2024beilinsonparshinadelessolidalgebraic}.
\[\begin{split}
A/I \otimes_{A_\bs} \bigl( \bigoplus_\N A \bigr)^\thicksim_{S_k} &\;\simeq\; A/I \otimes_{A_\bs} \uhom_{A_\bs} (\Ga_{S_k}, \oplus_\N A) \;\simeq \\
&\;\simeq\; \uhom_{(A/I)_\bs} (\Ga_{S_k} \otimes_A A/I, \oplus_\N A/I) \;\simeq \\
&\;\simeq\; \uhom_{(A/I)_\bs} (A/I, \oplus_\N A/I) \;\simeq \\
&\;\simeq\; \oplus_\N A/I \;,
\end{split}\] where, in the first step, we use that $A/I$ is pseudo compact as an $A_\bs$ module,
hence the functor $A/I \otimes_{A_\bs} - $ commutes with right t-bounded products
(\cite{brav2024beilinsonparshinadelessolidalgebraic}, Lemma 1.1.7), and that
$\Ga_{S_k}$ has finite projective dimension as an $A$-module 
(\cite{brav2024beilinsonparshinadelessolidalgebraic}, Proposition 1.6.2). In the second step,
we use that, because $A/I$ is $S_k$-torsion, we have an isomorphism $\Ga_{S_k} \otimes_A A/I \simeq A/I$.
\end{proof}

We now move on to the second claim of Proposition \ref{prop-main-tech-for-induction-step}.
We first note that, by Lemma \ref{lem-main-prop-part1} and Corollary \ref{cor-nuc-of-comp-is-idempotent}, it suffices to prove
that the functor \[
\nuc(A^\thicksim_{S_k}, A_\bs) \;\arr\; \nuc(L_kA, A_\bs)
\] admits a continuous fully faithful $\nuc(A^\thicksim_{S_k}, A_\bs)$-linear right adjoint. Indeed, if this adjoint exists,
then we can tensor it with $\nuc_{\le k}(A_\bs)$ over $\nuc(A^\thicksim_{S_k}, A_\bs)$, and, using the isomorphism \[
\nuc_{\le k}(A_\bs) \otimes_{\nuc(A^\thicksim_{S_k}, A_\bs)} \nuc(L_kA, A_\bs) \;\simeq\; \nuc_{\le k}(A_\bs) \otimes_{\nuc(A_\bs)} \nuc(L_kA, A_\bs) \;,
\] which holds because $\nuc(A^\thicksim_{S_k}, A_\bs)$ is idempotent over $\nuc(A_\bs)$,
we get the second claim of Proposition \ref{prop-main-tech-for-induction-step}.
\begin{lemma}\label{lem-main-prop-part2}
The functor \[
\nuc(A^\thicksim_{S_k}, A_\bs) \;\arr\; \nuc(L_kA, A_\bs)
\] admits a fully faithful continuous $\nuc(A^\thicksim_{S_k}, A_\bs)$-linear right adjoint.
\end{lemma}
\begin{proof}
If we exclude the linearity claim for a moment, then it suffices to prove that the right adjoint to the functor
\[
A^\thicksim_{S_k}-\mod_{A_\bs} \;\arr\; L_kA - \mod_{A_\bs}
\]
preserves nuclear objects.
By Corollary \ref{cor-nuc-over-adelic-and-rig}, it suffices to prove that the object \[
\uhom_{A_\bs} \bigl(\prod_{\N}A, L_kA \bigr) \;\simeq\; \prod_{\substack{\p\in \spec(A) \\ \dim(A/\p) = k}} \leftcomp\:\bigl( \bigoplus_\N A \bigr)_\p
\] is a nuclear $A^\thicksim_{S_k}$-module. The latter means that the map \[
\bigl( \bigoplus_\N A \bigr)^\thicksim_{S_k} \otimes_{A_\bs} \prod_{\substack{\p\in \spec(A) \\ \dim(A/\p) = k}} \leftcomp\:\bigl( \bigoplus_\N A \bigr)_\p \;\arr\; \prod_{\substack{\p\in \spec(A) \\ \dim(A/\p) = k}} \leftcomp\:\bigl( \bigoplus_{\N\times\N} A \bigr)_\p
\] is an isomorphism, which easily follows from Proposition A.2.1, \cite{brav2024beilinsonparshinadelessolidalgebraic}. The linearity part
also follows from that same map being an isomorphism.
\end{proof}

\begin{corollary}\label{cor-main-right-adj}
The inclusion \begin{equation}\label{cor-main-right-adj-eq}
\nuc_{\le k} (A_\bs) \;\arr\; \nuc_{\le k}(A_\bs) \otimes_{\nuc(A_\bs)} \nuc(L_kA, A_\bs)
\end{equation}
admits a fully faithful continuous $\nuc(A_\bs)$-linear right adjoint. 
\end{corollary}
\begin{proof}
By Lemma \ref{lem-main-prop-part1} and Corollary \ref{cor-nuc-of-comp-is-idempotent},
the functor (\ref{cor-main-right-adj-eq}) is identified with the functor \[
\nuc_{\le k} (A_\bs) \otimes_{\nuc(A^\thicksim_{S_k}, A_\bs)} \Bigl( \nuc(A^\thicksim_{S_k}, A_\bs) \;\arr\; \nuc(L_kA, A_\bs) \Bigr) \;.
\] The claim now follows from Lemma \ref{lem-main-prop-part2}.
\end{proof}

We now prove the third and final claim of Proposition \ref{prop-main-tech-for-induction-step}, packaged into the following lemma.
\begin{lemma}
The functor \[
\nuc_{\le k-1} (A_\bs) \;\arr\; \nuc_{\le k} (A_\bs)
\] factors over the kernel \[
\mathrm{Ker} \Bigl( \nuc_{\le k} (A_\bs) \;\arr\; \nuc_{\le k}(A_\bs) \otimes_{\nuc(A_\bs)} \nuc(L_kA, A_\bs) \Bigr) \;.
\]
The resulting functor
\[
\nuc_{\le k-1}(A_\bs) \;\arr\; \mathrm{Ker} \Bigl( \nuc_{\le k} (A_\bs) \;\arr\; \nuc_{\le k}(A_\bs) \otimes_{\nuc(A_\bs)} \nuc(L_kA, A_\bs) \Bigr)
\] is an equivalence.
\end{lemma}

\begin{proof}
Because we a have commutative diagram \[\xymatrix{
A^\thicksim_{S_{k}} - \mod_{A_\bs} \ar@{^(->}[rr] && \mod_{A_\bs} \ar[rr] && \Ga(S_{k-1}^c, \O)_\bullet - \mod \\
\nuc_{\le k}(A_\bs) \ar@{^(->}[u] \ar@{^(->}[rr] && \nuc(A_\bs) \ar[rr]\ar@{^(->}[u] && \nuc(\Ga(S_{k-1}^c, \O)_\bullet) \ar@{^(->}[u]
}\]
and the isomorphism \[
(A^\thicksim_{S_{k}}, A_\bs) \otimes_{A_\bs} \Ga(S_{k-1}^c, \O)_\bullet \;\simeq\; (L_{k}A)_\bullet \;,
\]
the image of the composite \[
\nuc_{\le k} (A_\bs) \;\arr\; \nuc(A_\bs) \;\arr\; \nuc(\Ga(S_{k-1}^c, \O)_\bullet) \;\subset\; \Ga(S_{k-1}^c, \O)_\bullet - \mod 
\] lies inside the full subcategory \[
(L_{k}A)_\bullet - \mod \;\simeq\; L_{k}A - \mod \bigl( \Ga(S_{k-1}^c, \O)_\bullet - \mod \bigr) \;\subseteq\; \Ga(S_{k-1}^c, \O)_\bullet - \mod \;.
\] It then follows that the kernel \[
\nuc_{\le k-1} (A_\bs) \;\simeq\; \mathrm{Ker} \Bigl( \nuc_{\le k} (A_\bs) \;\arr\; \nuc(\Ga(S_{k-1}^c, \O)_\bullet) \Bigr)
\] can be computed as kernel of the map \[
\nuc_{\le k} (A_\bs) \;\arr\; \nuc((L_{k}A)_\bullet) \;.
\]
Moreover, by the adjunction, the latter factors as the composite
\[
\nuc_{\le k} (A_\bs) \;\arr\; \nuc_{\le k} (A_\bs) \otimes_{\nuc(A_\bs)} \nuc((L_{k}A)_\bullet) \;\arr\; \nuc((L_{k}A)_\bullet) \;,
\]
and the functor \[
\nuc_{\le k} (A_\bs) \otimes_{\nuc(A_\bs)} \nuc((L_{k}A)_\bullet) \;\arr\; \nuc((L_{k}A)_\bullet)
\] is fully faithful because the category $\nuc((L_{k}A)_\bullet)$ is dualizable over $\nuc(A_\bs)$.
It remains to recall that we have an equivalence \[
\nuc(L_kA, A_\bs) \;\simeq\; \nuc((L_kA)_\bullet)
\] by Corollary \ref{cor-nuc-of-a-completion}.
\end{proof}

\ssec{Adelic descent for nuclear sheaves}
For completeness, here we state and prove a `categorical' version of Theorem \ref{thm-nuc-adelic-descent}.
We note that the second claim can also be deduced from Theorem 4.2.4 of \cite{brav2024beilinsonparshinadelessolidalgebraic}
by applying rigidification.
\begin{theorem}\label{thm-adelic-descent-cats}
The functors \[
\nuc(X_\bs) \;\arr\; \stackrel[\substack{0 \le k_0 < \ldots < k_m \le d \\ m \ge 0}]{}{\lim} \nuc(L_{k_0} \ldots L_{k_m} \O_X, X_\bs)
\] and \[
\nuc(X_\bs) \;\arr\; \stackrel[\substack{0 \le k_0 < \ldots < k_m \le d \\ m \ge 0}]{\dual}{\lim} \nuc(L_{k_0} \ldots L_{k_m} \O_X, X_\bs)
\] are equivalences. In particular, the forgetful functor $\pr^\dual_\st \arr \pr^L_\st$ preserves this limit.
\end{theorem}
By the same reduction as in section \ref{ssec-reduction-to-affine} combined with Proposition 1.87 of
\cite{efimov2025ktheorylocalizinginvariantslarge}, it suffices to prove the statement in the affine case. We prove the following
proposition by the induction on $k$.
\begin{proposition}
For any $k \in 0..d$, the natural functors \[
\nuc_{\le k} (A_\bs) \;\arr\; \stackrel[\substack{0 \le k_0 < \ldots < k_m \le k \\ m \ge 0}]{}{\lim} \nuc_{\le k}(A_\bs) \otimes_{\nuc(A_\bs)} \nuc(L_{k_0} \ldots L_{k_m} A, A_\bs)
\] and \[
\nuc_{\le k} (A_\bs) \;\arr\; \stackrel[\substack{0 \le k_0 < \ldots < k_m \le k \\ m \ge 0}]{\dual}{\lim} \nuc_{\le k}(A_\bs) \otimes_{\nuc(A_\bs)} \nuc(L_{k_0} \ldots L_{k_m} A, A_\bs)
\] are equivalences. In particular, the forgetful functor $\pr^\dual_\st \arr \pr^L_\st$ preserves this limit.
\end{proposition}
\begin{proof} In the base case---$k = 0$---both claims say that the map \[
\nuc_{\le 0}(A_\bs) \;\arr\; \nuc_{\le 0}(A_\bs) \otimes_{A_\bs} \nuc(L_0A, A_\bs)
\] is an equivalence. Because $A^\thicksim_{S_0} = L_0A$, that immediately follows from the first claim of
Proposition \ref{prop-main-tech-for-induction-step} combined with Corollary \ref{cor-nuc-of-comp-is-idempotent}.

Now assume that both claims hold for some $k \ge 0$. Let us denote the limit \[
\stackrel[\substack{0 \le k_0 < \ldots < k_m \le k \\ m \ge 0}]{}{\lim} \nuc(L_{k_0} \ldots L_{k_m} A, A_\bs) 
\] by $\K$. It suffices to prove that the square
\[\xymatrix{
\nuc_{\le k+1}(A_\bs) \ar[rr]\ar[d] && \nuc_{\le k+1}(A_\bs) \otimes_{\nuc(A_\bs)} \nuc(L_{k+1}A, A_\bs) \ar[d] \\
\nuc_{\le k+1}(A_\bs) \otimes_{\nuc(A_\bs)} \K \ar[rr] && \nuc_{\le k+1}(A_\bs) \otimes_{\nuc(A_\bs)} \nuc(L_{k+1}A, A_\bs) \otimes_{\nuc(A_\bs)} \K
}\] is a pull-back square $\pr^L_\st$. By Proposition \ref{prop-main-tech-for-induction-step},
both of the horizontal arrows are strongly continuous localizations.
It therefore suffices to prove that the kernels of the horizontal arrows agree, which follows
from the last claim of Proposition \ref{prop-main-tech-for-induction-step} and the induction hypothesis.
The case of the dualizable limit is handled identically except we additionally use Proposition 1.87 of
\cite{efimov2025ktheorylocalizinginvariantslarge} to argue that the above square is also a pull-back square in $\pr^\dual_\st$.
\end{proof}

\ssec{General filtration by specializing subsets}
In fact, the relevant part of the argument almost without modification applies to the following more general situation.
Given a filtration \[
\emptyset = T_{-1} \subset T_0 \subset T_1 \subset \ldots \subset T_n = X 
\] by specialization closed subsets, we get a cubical diagram \begin{equation}\label{gen-filtr-cubical-diag}
0 \le k_0 < k_1 < \ldots < k_m \le n \;\mapsto\; \nuc\Bigl( (\O_X)^\thicksim_{T_{k_0}, T^c_{k_0-1}} \otimes_{X_\bs} \ldots \otimes_{X_\bs} (\O_X)^\thicksim_{T_{k_m}, T^c_{k_m-1}}, X_\bs \Bigr) \;.
\end{equation}
\begin{theorem}
The cube (\ref{gen-filtr-cubical-diag}) is a limit cube in both $\pr^L_\st$ and $\pr^\dual_\st$, and is mapped to a limit cube
by any stable localizing invariant.
\end{theorem}
However, in general, we do not know how to compute the tensor products appearing in this diagram. It also might be worth noting that it is
unclear to the author whether an analog of Theorem 4.2.4 of \cite{brav2024beilinsonparshinadelessolidalgebraic} holds in this generality.

\bibliographystyle{IEEEtran}
\bibliography{refs}{}

\bigskip
\bigskip

\noindent Grigorii~Konovalov, {\sc Center for Advanced Studies, Skoltech, Moscow; Centre of Pure Mathematics, MIPT, Moscow;
Faculty of Mathematics, National Research University Higher School of Economics, Moscow;}
\href{mailto:grisha.v.konovalov@gmail.com}{grisha.v.konovalov@gmail.com}

\end{document}